\documentclass[a4paper,10pt]{amsart}
\usepackage[utf8]{inputenc}
\usepackage{amsxtra}
\usepackage{amsopn}
\usepackage{amsmath,amsthm,amssymb}
\usepackage{amscd}
\usepackage{mathtools}
\usepackage{amsfonts}
\usepackage{latexsym}
\usepackage{verbatim}
\usepackage{MnSymbol}
\usepackage{xcolor}
\usepackage{cases}
\usepackage{tikz-cd}
\usepackage{cases}

\newcommand{\lr}{\longrightarrow}

\newcommand{\la}{\llangle}
\newcommand{\ra}{\rrangle}

\let\c\overline
\theoremstyle{plain}
\newtheorem{theorem}{Theorem}[section]
\newtheorem*{theorem*}{Theorem \ref{thm-main}}
\newtheorem*{theorem**}{Theorem \ref{thm-dol-11-dim-char}}
\newtheorem*{proposition*}{Proposition \ref{prop-alm-kahl}}

\newtheorem{lemma}[theorem]{Lemma}

\newtheorem{corollary}[theorem]{Corollary}
\newtheorem{remark}[theorem]{Remark}
\newtheorem{example}[theorem]{Example}

\newtheorem*{question*}{Question}
\newtheorem*{mt*}{Main Theorem}
\makeatletter
\newtheorem*{rep@theorem}{\rep@title}
\newcommand{\newreptheorem}[2]{%
\newenvironment{rep#1}[1]{%
 \def\rep@title{#2 \ref{##1}}%
 \begin{rep@theorem}}%
 {\end{rep@theorem}}}
\makeatother

\newreptheorem{theorem}{Theorem}
\newreptheorem{corollary}{Corollary}

\newcommand\C{{\mathbb C}}

\newcommand\N{{\mathbb N}}
\newcommand\R{{\mathbb R}}
\newcommand\Z{{\mathbb Z}}

\newcommand{\del}{{\partial}}
\newcommand{\delbar}{{\overline{\del}}}

\newcommand{\cinf}{\mathcal{C}^\infty}
\renewcommand{\H}{\mathcal{H}}

\DeclareMathOperator{\vol}{Vol}

\DeclareMathOperator{\End}{End}
\DeclareMathOperator{\id}{id}

\let\c\overline
\let\phi\varphi

\title[Primitive decomposition of BC and Dolbeault harmonic $(k,k)$-forms]{Primitive decomposition of Bott-Chern and Dolbeault harmonic $(k,k)$-forms on compact almost K\"ahler manifolds}
\author{Tom Holt}
\address{Dipartimento di Scienze Matematiche, Fisiche e Informatiche\\
Unit\`{a} di Matematica e Informatica\\
Universit\`{a} degli Studi di Parma\\
Parco Area delle Scienze 53/A \\
43124 Parma, Italy}
\email{thomas.holt@warwick.ac.uk}
\author{Riccardo Piovani}
\address{Dipartimento di Scienze Matematiche, Fisiche e Informatiche\\
Unit\`{a} di Matematica e Informatica\\
Universit\`{a} degli Studi di Parma\\
Parco Area delle Scienze 53/A \\
43124 Parma, Italy}
\email{riccardo.piovani@unipr.it}

\thanks{\newline 
The second author is partially supported by GNSAGA of INdAM}

\begin{document}
\maketitle

\begin{abstract} 
We consider the primitive decomposition of $\overline \partial, \partial$, Bott-Chern and Aeppli-harmonic $(k,k)$-forms on compact almost K\"ahler manifolds $(M,J,\omega)$. For any $D \in \{\overline\partial, \partial, BC, A\}$, we prove that the $L^k P^0$ component of $\psi \in \mathcal{H}_{D}^{k,k}$, is a constant multiple of $\omega^k$. 
Focusing on dimension 8, we give a full description of the spaces $\mathcal{H}_{BC}^{2,2}$ and $\mathcal{H}_{A}^{2,2}$, from which follows $\H^{2,2}_{BC}\subseteq\H^{2,2}_{\del}$ and $\H^{2,2}_{A}\subseteq\H^{2,2}_{\delbar}$.
We also provide an almost K\"ahler 8-dimensional example where the previous inclusions are strict and the primitive components of an harmonic form $\psi \in \mathcal{H}_{D}^{k,k}$ are not $D$-harmonic, showing that the primitive decomposition of $(k,k)$-forms in general does not descend to harmonic forms.
\end{abstract}

\section{Introduction}\label{introduction}

A recent answer to a question of Kodaira and Spencer, \cite[Problem 20]{Hi}, shows that the dimension of the space of Dolbeault harmonic forms depends on the choice of the metric on a given compact almost complex manifold, see \cite{HZ,HZ2}.

The primitive decomposition of harmonic forms has proven to be useful in describing the spaces of harmonic $(1,1)$-forms in dimension 4. In the case of Dolbeault harmonic forms it has been used to show that $h^{1,1}_\delbar:=\dim_\C\H^{1,1}_\delbar$ is either equal to $b^-$ or $b^-+1$, depending on the choice of metric, see \cite{Ho,HZ,TT}. Similarly, for Bott-Chern harmonic forms, it yields $h^{1,1}_{BC}:=\dim_\C\H^{1,1}_{BC}=b^-+1$ for all metrics, see \cite{Ho,PT4}. See \cite{piovani-invariant,PT5,tardini-tomassini-dim6} for other related results and \cite{piovani-parma,zhang-parma} for two surveys on the subject. 

In this paper, we explore what the primitive decomposition can tell us about harmonic $(k,k)$-forms in higher dimensions. 
We start by considering a $2n$-dimensional almost Hermitian manifold $(M,J,\omega)$. The almost complex structure $J$ induces the bidegree decomposition on the space of complex valued $k$-forms 
$$A^k_{\mathbb{C}} = \bigoplus_{p+q = k} A^{p,q}. $$
Additionally, the almost Hermitian structure induces the primitive decomposition on the space of $k$-forms given by
$$A^k=\bigoplus_{r\geq\max(k-n,0)}L^r(P^{k-2r}),$$
where $L:=\omega\wedge$, $\Lambda:=*^{-1}L*$ and $P^s := \ker \Lambda \cap A^s$ is the space of primitive $s$-forms, for $s\le n$ (see e.g., \cite[p. 26, Th\'eor\`eme 3]{weil}). These two decompositions are compatible with each other.

In fact, for K\"ahler manifolds, i.e., when $J$ is integrable and $d\omega = 0$, the primitive decomposition passes to the space of $d$-harmonic $(p,q)$-forms, denoted by $\mathcal{H}^{p,q}_{d}(M,J) := \ker \Delta_{d}\cap A^{p,q}$, namely
\begin{equation}\label{decomp of d harmonic space}
    \mathcal{H}_d^{p,q}=\bigoplus_{r\geq\max(p+q-n,0)}L^r(
    \mathcal{H}_d^{p-r,q-r}\cap P^{p-r,q-r}).
\end{equation}
where $P^{p,q}:= P^{m}_{\mathbb C}\cap A^{p,q}$.

On K\"ahler manifolds, we also know that $\mathcal{H}_{d}^{p,q} = \mathcal{H}_{D}^{p,q}$ for all $D \in \{\overline \partial, \partial, BC, A\}$ (see Section 2 for the definitions of these spaces), therefore we have  
\begin{equation}\label{decomp of D harmonic spaces}
\mathcal{H}_D^{p,q}=\bigoplus_{r\geq\max(p+q-n,0)}L^r(
\mathcal{H}_D^{p-r,q-r}\cap P^{p-r,q-r}).    
\end{equation}
We remark that \eqref{decomp of d harmonic space} and \eqref{decomp of D harmonic spaces} have a cohomological meaning in the K\"ahler setting.

In \cite[Corollary 5.4]{cirici-wilson-2}, Cirici and Wilson prove that \eqref{decomp of d harmonic space} continues to hold true for almost K\"ahler manifolds, however it does not directly follow that \eqref{decomp of D harmonic spaces} must also be true. Cattaneo, Tardini and Tomassini prove in \cite[Theorem 3.4 and Corollary 3.5]{cattaneo-tardini-tomassini} that:
\begin{theorem}\label{thm CTT}
Let $(M,J,\omega)$ be a compact $2n$-dimensional almost K\"ahler manifold, then the following decompositions hold
$$\mathcal{H}_{\overline\partial}^{1,1} = \mathbb{C}\, \omega \oplus \left(\mathcal{H}_{\overline\partial}^{1,1}\cap P^{1,1}\right), $$
$$\mathcal{H}_{\partial}^{1,1} = \mathbb{C}\, \omega \oplus \left(\mathcal{H}_{\partial}^{1,1}\cap P^{1,1}\right), $$
$$\mathcal{H}_{\overline\partial}^{n-1,n-1} = \mathbb{C}\, \omega^{n-1} \oplus L^{n-2}\left(\mathcal{H}_{\partial}^{1,1}\cap P^{1,1}\right), $$
$$\mathcal{H}_{\partial}^{n-1,n-1} = \mathbb{C}\, \omega^{n-1} \oplus L^{n-2}\left(\mathcal{H}_{\overline\partial}^{1,1}\cap P^{1,1}\right). $$
\end{theorem}
This means that, on almost K\"ahler manifolds, $\mathcal{H}_{\overline \partial}^{p,q}$ and $\mathcal{H}_{\partial}^{p,q}$ both have primitive decompositions when $(p,q) = (1,1)$ and, applying the Hodge $*$ operator to the $(1,1)$-decompositions, when $(p,q) = (n-1,n-1)$.

In \cite[Theorems 3.2 and 3.3]{piovani-tardini}, Tardini and the second author prove the following results:
\begin{theorem}\label{thm PT}
Let $(M, J, \omega)$ be a compact $2n$-dimensional almost K\"ahler manifold, then the following decompositions hold
$$\mathcal{H}_{BC}^{1,1} = \mathbb{C}\, \omega \oplus \left(\mathcal{H}_{BC}^{1,1}\cap P^{1,1}\right), $$
$$\mathcal{H}_{A}^{1,1} = \mathbb{C}\, \omega \oplus \left(\mathcal{H}_{A}^{1,1}\cap P^{1,1}\right), $$
$$\mathcal{H}_{BC}^{n-1,n-1} = \mathbb{C}\, \omega^{n-1} \oplus L^{n-2}\left(\mathcal{H}_{A}^{1,1}\cap P^{1,1}\right), $$
$$\mathcal{H}_{A}^{n-1,n-1} = \mathbb{C}\, \omega^{n-1} \oplus L^{n-2}\left(\mathcal{H}_{BC}^{1,1}\cap P^{1,1}\right). $$
\end{theorem}
We therefore see that, in the almost K\"ahler setting, $\mathcal{H}_{BC}^{p,q}$ and $\mathcal{H}_{A}^{p,q}$ both have primitive decompositions when $(p,q) = (1,1)$ or $(n-1,n-1)$.

These two results are sufficient to prove that either \eqref{decomp of D harmonic spaces} or its dual through the Hodge $*$ operator hold for any space of $D$-harmonic $(k,k)$-forms on any compact almost K\"ahler manifold with dimension up to 6. 
This raises the following question, which we shall answer in this paper: does \eqref{decomp of D harmonic spaces} (or its $*$ dual) hold for $(k,k)$-forms in general for compact almost K\"ahler manifolds with dimension 8 or greater?
We note that \eqref{decomp of D harmonic spaces} has been shown to fail for dimension 6 in bidegree $(2,1)$ in \cite[Proposition 5.1]{cattaneo-tardini-tomassini} for $D \in \{\overline\partial, \partial\}$ and in \cite[Proposition 5.1]{piovani-tardini} for $D \in \{BC, A\}$.

We also remark that the almost K\"ahler assumption is necessary for this kind of primitive harmonic decomposition. To see that this is the case in dimension 4 we refer the reader to \cite{TT,PT4}.

The structure of this paper is as follows. In Section 2 we give a brief overview of some of the basic results which will be used throughout the paper.   
In Section 3 we show that Theorem \ref{thm PT} may be partially extended to $(k,k)$-forms.
\begin{reptheorem}{cor bc decomp}
Let $(M,J,\omega)$ be a compact almost K\"ahler manifold of real dimension $2n$. For any $k\in \mathbb{N}$ we have
$$\mathcal{H}^{k,k}_{BC} = \mathbb{C}\, \omega^k \oplus \left( \mathcal{H}_{BC}^{k,k} \cap \ker L^{n-k} \right)$$
and
$$\mathcal{H}^{k,k}_{A} = \mathbb{C}\, \omega^k \oplus \left( \mathcal{H}_{A}^{k,k} \cap \ker L^{n-k} \right).$$
\end{reptheorem}
We also consider the 8-dimensional case in more detail, yielding the following description.
\begin{repcorollary}{cor bc 8}
Let $(M,J,\omega)$ be a compact almost K\"ahler manifold of real dimension $8$. We have
\[
\H^{2,2}_{BC}=\C\,\omega^2\oplus\left\{\omega\wedge\alpha+\beta\,|\,\alpha\in P^{1,1},\,\beta\in P^{2,2},\,\omega\wedge\del\alpha+\del\beta=\delbar\alpha=\delbar\beta=0\right\},
\]
and
\[
\H^{2,2}_{A}=\C\,\omega^2\oplus\left\{\omega\wedge\alpha+\beta\,|\,\alpha\in P^{1,1},\,\beta\in P^{2,2},\,\omega\wedge\del\alpha-\del\beta=\delbar\alpha=\delbar\beta=0\right\}.
\]
\end{repcorollary}

In Section 4 we show that Theorem \ref{thm CTT} may also be partially extended to $(k,k)$-forms.
\begin{reptheorem}{cor Dolbeault partial decomp}
Let $(M,J,\omega)$ be a compact almost K\"ahler manifold of real dimension $2n$. For any $k\in \mathbb{N}$ we have
$$\mathcal{H}^{k,k}_{\overline \partial} = \mathbb{C}\, \omega^k \oplus \left( \mathcal{H}_{\overline\partial}^{k,k} \cap \ker L^{n-k} \right)$$
and
$$\mathcal{H}^{k,k}_{\partial} = \mathbb{C}\, \omega^k \oplus \left( \mathcal{H}_{\partial}^{k,k} \cap \ker L^{n-k} \right).$$
\end{reptheorem}
Looking at the special case of this corollary in dimension 8, along with Corollary \ref{cor bc 8}, we are able to deduce the following
\begin{repcorollary}{cor incl}
Let $(M,J,\omega)$ be a compact almost K\"ahler manifold of real dimension $8$. We have
\[
\H^{2,2}_{BC}\subseteq\H^{2,2}_{\del},\ \ \ \ \ \ \ \H^{2,2}_{A}\subseteq\H^{2,2}_{\delbar}.
\]
\end{repcorollary}

Finally in Section 5, we consider a non left invariant almost K\"ahler structure on the 8-dimensional torus $\mathbb{T}^8=\Z^8\backslash \R^8$.
We use this example to show that there exists a $(2,2)$-form contained in $\mathcal{H}^{2,2}_{\partial}$ but not in $\mathcal{H}^{2,2}_{BC}$, and likewise there exists a $(2,2)$-form in $\mathcal{H}^{2,2}_{\overline\partial}$ but not in $\mathcal{H}^{2,2}_{A}$.  
We also show that there exists a form $\psi = \omega \wedge \alpha + \beta  \in \mathcal{H}^{2,2}_{BC}$ whose components with respect to the primitive decomposition, $\alpha  \in P^{1,1}, \beta \in P^{2,2}$, are not themselves Bott-Chern harmonic. From this we can conclude that the primitive decomposition does not in general apply to $D$-harmonic $(2,2)$-forms in dimension 8, for $D\in \{BC,A,\delbar,\del\}$.
\begin{repcorollary}{cor counterexample}
There exists a compact almost K\"ahler manifold $(M,J,\omega)$ of real dimension $8$ such that
\[
\H^{2,2}_{BC}\not\supseteq\H^{2,2}_{\del},\ \ \ \ \ \ \ \H^{2,2}_{A}\not\supseteq\H^{2,2}_{\delbar}
\]
and
\[
\H^{2,2}_{D}\not\subseteq \C\,\omega^2\oplus L\left(P^{1,1}\cap\H^{1,1}_{D}\right)\oplus \left(P^{2,2}\cap \H^{2,2}_{D}\right),
\]
where $D\in\{BC,A,\delbar,\del\}$.
\end{repcorollary}

We also consider another 8-dimensional compact nilmanifold, focusing on the subspace of left invariant harmonic forms in $\mathcal{H}^{2,2}_{D}$. We show that these spaces satisfy \eqref{decomp of D harmonic spaces} for all $D \in \{\overline \partial, \partial, BC,A\}$. Furthermore, we show that in this example these spaces are in fact all equal and have dimension 16.


\section{Preliminaries}\label{preliminaries}

Throughout this paper, we will only consider connected manifolds without boundary.
Let $(M,J)$ be an almost complex manifold of dimension $2n$, i.e., a $2n$-differentiable manifold endowed with an almost complex structure $J$, that is $J\in\End(TM)$ and $J^2=-\id$. The complexified tangent bundle $T_{\C}M=TM\otimes\C$ decomposes into the two eigenspaces of $J$ associated to the eigenvalues $i,-i$, which we denote respectively by $T^{1,0}M$ and $T^{0,1}M$, giving us
\begin{equation*}
T_{\C}M=T^{1,0}M\oplus T^{0,1}M.
\end{equation*}
Denoting by $\Lambda^{1,0}M$ and $\Lambda^{0,1}M$ the dual vector bundles of $T^{1,0}M$ and $T^{0,1}M$, respectively, we set
\begin{equation*}
\Lambda^{p,q}M=\bigwedge^p\Lambda^{1,0}M\wedge\bigwedge^q\Lambda^{0,1}M
\end{equation*}
to be the vector bundle of $(p,q)$-forms, and let $A^{p,q}=\Gamma(M,\Lambda^{p,q}M)$ be the space of smooth sections of $\Lambda^{p,q}M$. We denote by $A^k=\Gamma(M,\Lambda^{k}M)$ the space of $k$-forms. Note that $\Lambda^{k}M\otimes\C=\bigoplus_{p+q=k}\Lambda^{p,q}M$.

Let $f\in\cinf(M,\C)$ be a smooth function on $M$ with complex values. Its differential $df$ is contained in $A^1\otimes\C=A^{1,0}\oplus A^{0,1}$. On complex 1-forms, the exterior derivative acts as
\[
d:A^1\otimes\C\to A^2\otimes\C=A^{2,0}\oplus A^{1,1}\oplus A^{0,2}.
\]
 Therefore, it turns out that the derivative operates on $(p,q)$-forms as
\begin{equation*}
d:A^{p,q}\to A^{p+2,q-1}\oplus A^{p+1,q}\oplus A^{p,q+1}\oplus A^{p-1,q+2},
\end{equation*}
where we denote the four components of $d$ by
\begin{equation*}
d=\mu+\del+\delbar+\c\mu.
\end{equation*}
From the relation $d^2=0$, we derive
\begin{equation*}
\begin{cases}
\mu^2=0,\\
\mu\del+\del\mu=0,\\
\del^2+\mu\delbar+\delbar\mu=0,\\
\del\delbar+\delbar\del+\mu\c\mu+\c\mu\mu=0,\\
\delbar^2+\c\mu\del+\del\c\mu=0,\\
\c\mu\delbar+\delbar\c\mu=0,\\
\c\mu^2=0.
\end{cases}
\end{equation*}
We also define the operator $d^c:=J^{-1}dJ$. It is a straightforward computation to show that
\[
d^c=i(\mu-\del+\delbar-\c\mu).
\]

If the almost complex structure $J$ is induced from a complex manifold structure on $M$, then $J$ is called integrable. Recall that $J$ is integrable if and only if the exterior derivative decomposes into $d=\del+\delbar$.

A Riemannian metric $g$ on $M$ which is preserved by $J$, \textit{i.e.} $g(J\cdot, J \cdot) = g(\cdot, \cdot)$, is called almost Hermitian.
Let $g$ be an almost Hermitian metric, the $2$-form $\omega$ such that
\begin{equation*}
\omega(u,v)=g(Ju,v)\ \ \forall u,v\in\Gamma(TM)
\end{equation*}
is called the fundamental form of $g$. We will call $(M,J,\omega)$ an almost Hermitian manifold.
We denote by $h$ the Hermitian extension of $g$ on the complexified tangent bundle $T_\C M$, and by the same symbol $g$ the $\C$-bilinear symmetric extension of $g$ on $T_\C M$. Also denote by the same symbol $\omega$ the $\C$-bilinear extension of the fundamental form $\omega$ of $g$ on $T_\C M$. 
Thanks to the elementary properties of the two extensions $h$ and $g$, we may want to consider $h$ as a Hermitian operator
$T^{1,0}M\times T^{1,0}M\to\C$ and $g$ as a $\C$-bilinear operator $T^{1,0}M\times T^{0,1}M\to\C$.
Note that $h(u,v)=g(u,\overline{v})$ for all $u,v\in \Gamma(T^{1,0}M)$.

Let $(M,J,\omega)$ be an almost Hermitian manifold of real dimension $2n$. Denote the extension of $h$ to $(p,q)$-forms by the Hermitian inner product $\langle\cdot,\cdot\rangle$.
Let $*:A^{p,q}\lr A^{n-q,n-p}$ be the $\C$-linear extension of the standard Hodge $*$ operator on Riemannian manifolds with respect to the volume form $\vol=\frac{\omega^n}{n!}$, i.e., $*$ is defined by the relation
\[
\alpha\wedge{*\c\beta}=\langle\alpha,\beta\rangle\vol\ \ \ \forall\alpha,\beta\in A^{p,q}.
\]
Integrating the pointwise Hermitian inner product on the manifold, we get the standard $L^2$ product here denoted by
\[
\la\alpha,\beta\ra=\int_M\langle\alpha,\beta\rangle\vol\ \ \ \forall\alpha,\beta\in A^{p,q},
\]
which is surely well defined if $M$ is compact.
Then the operators
\begin{equation*}
d^*=-*d*,\ \ \ \mu^*=-*\c\mu*,\ \ \ \del^*=-*\delbar*,\ \ \ \delbar^*=-*\del*,\ \ \ \c\mu^*=-*\mu*,
\end{equation*}
are the $L^2$ formal adjoint operators respectively of $d,\mu,\del,\delbar,\c\mu$. Recall that 
\[
\Delta_{d}=dd^*+d^*d
\]
 is the Hodge Laplacian, and, as in the integrable case, set 
\begin{equation*}
\Delta_{\del}=\del\del^*+\del^*\del,\ \ \ \Delta_{\delbar}=\delbar\delbar^*+\delbar^*\delbar,
\end{equation*}
respectively as the $\del$ and $\delbar$ Laplacians. Again, as in the integrable case, set
\begin{equation*}
\Delta_{BC}=
\del\delbar\delbar^*\del^*+
\delbar^*\del^*\del\delbar+\del^*\delbar\delbar^*\del+\delbar^*\del\del^*\delbar
+\del^*\del+\delbar^*\delbar,
\end{equation*}
and
\begin{equation*}
\Delta_{A}= \del\delbar\delbar^*\del^*+
\delbar^*\del^*\del\delbar+
\del\delbar^*\delbar\del^*+\delbar\del^*\del\delbar^*+
\del\del^*+\delbar\delbar^*,
\end{equation*}
respectively as the Bott-Chern and the Aeppli Laplacians. Note that
\begin{equation}\label{bc-a-duality}
*\Delta_{BC}=\Delta_{A}*\ \ \ \Delta_{BC}*=*\Delta_{A}.
\end{equation}

If $M$ is compact, then we easily deduce the following relations
\begin{equation}\label{eq-char-harm-forms}
\begin{cases}
\Delta_{d}=0\ &\iff\ d=0,\ d*=0,\\
\Delta_{\del}=0\ &\iff\ \del=0,\ \delbar*=0,\\
\Delta_{\delbar}=0\ &\iff\ \delbar=0,\ \del*=0,\\
\Delta_{BC}=0\ &\iff \del=0,\ \delbar=0,\ \del\delbar*=0,\\
\Delta_{A}=0\ &\iff \del*=0,\ \delbar*=0,\ \del\delbar=0,
\end{cases}
\end{equation}
which characterize the spaces of harmonic forms
\begin{equation*}
\H^{k}_{d},\ \ \ \H^{p,q}_{\del},\ \ \ \H^{p,q}_{\delbar},\ \ \  \H^{p,q}_{BC},\ \ \ \H^{p,q}_{A},
\end{equation*}
defined as the spaces of forms which are in the kernel of the associated Laplacians.
All these Laplacians are elliptic operators on the almost Hermitian manifold $(M,J,\omega)$ (cf. \cite{Hi}, \cite{PT4}), implying that all the spaces of harmonic forms are finite dimensional when the manifold is compact. 

Now we introduce some notation and recall some well known facts about primitive forms.
We denote by
$$
L:\Lambda^kM\to\Lambda^{k+2}M\,,\quad \alpha\mapsto\omega\wedge\alpha
$$
the Lefschetz operator and by
$$
\Lambda:\Lambda^kM\to\Lambda^{k-2}M\,,\quad \Lambda=*^{-1}L*
$$
its adjoint.
A differential $k$-form $\alpha$ on $M$, for $k\leq n$, is said to be {\em primitive} if $\Lambda\alpha=0$, or equivalently if 
$$
L^{n-k+1}\alpha=0.
$$
 Then we have the following vector bundle decomposition (see e.g., \cite[p. 26, Th\'eor\`eme 3]{weil})
\begin{equation}\label{eq-prim-dec-forms}
\Lambda^kM=\bigoplus_{r\geq\max(k-n,0)}L^r(P^{k-2r}M),
\end{equation}
where we use
$$
P^{s}M:=\ker\big(\Lambda:\Lambda^{s}M\to\Lambda^{s-2}M\big)
$$
to denote the bundle of primitive $s$-forms. 
For any given $\beta\in P^kM$, we have the following formula (cf. \cite[p. 23, Th\'eor\`eme 2]{weil}) involving the Hodge $*$ operator and the Lefschetz operator
\begin{equation}\label{*-primitive}
*L^r\beta=(-1)^{\frac{k(k+1)}{2}}\frac{r!}{(n-k-r)!}L^{n-k-r}J\beta.
\end{equation}
We recall that the map $L^h:\Lambda^kM\to\Lambda^{k+2h}M$ is injective for $h+k\le n$ and is surjective for $h+k\ge n$. 

Furthermore, the decomposition above is compatible with the bidegree decomposition on the bundle of complex $k$-forms $\Lambda_\C^kM$ induced by $J$, that is 
$$
P_\C^kM=\bigoplus_{p+q=k}P^{p,q}M,
$$
where 
$$
P^{p,q}M=P^k_\C M\cap\Lambda^{p,q}M.
$$
In fact, we have
\begin{equation}\label{eq-prim-dec-forms-2}
\Lambda^{p,q}M=\bigoplus_{r\geq\max(p+q-n,0)}L^r(P^{p-r,q-r}M).
\end{equation}
Finally, let us set $P^s:=\Gamma(M,P^sM)$ and $P^{p,q}:=\Gamma(M,P^{p,q}M)$.

\section{Primitive decomposition of Bott-Chern harmonic $(k,k)$-forms}
In order to prove our main result, we will need the following lemmas.
The next one is well known, see for instance \cite[Theorem 3.2]{piovani-tardini} or \cite[Theorem 4.3]{PT4}. We include an outline of the proof here for the convenience of the reader.
\begin{lemma}\label{lemma elliptic}
Let $(M,J,\omega)$ be a compact almost K\"ahler manifold of real dimension $2n$. Let $f\in\cinf(M,\C)$ be a smooth complex valued function. If
\[
\omega^{n-1}\wedge\del\delbar f=0,
\]
then $f\in\C$ is a complex constant.
\begin{proof}
Let $V_1, \dots, V_n$ be a local frame of $T^{1,0}M$, with $\phi^1, \dots, \phi^n$ the dual coframe of $\Lambda^{1,0}M$, chosen such that
$\omega =  i \sum_{j} \phi^{j \overline j}. $
We can then write 
\begin{align*}
    \partial \overline \partial f &= \partial \left(\sum_{j=1}^{n} \overline{V_j}f \phi^{\overline j}\right)\\
    &= \sum_{i,j= 1}^{n} V_i \overline{V_j} f \phi^{i\overline j} + \sum_{j=1}^{n}   \overline{V_j} f \partial\phi^{\overline j}.
\end{align*}
The wedge product of $\omega^{n-1}$ with $\phi^{i\overline j}$ is zero unless $i=j$, therefore
\begin{align*}
    \omega^{n-1}\wedge \partial \overline \partial f &= -i(n-1)!\sum_{j=1}^{n} V_j \overline{V_j} f \vol + R(f) \vol
\end{align*}
where $R$ is a differential operator involving at most first order derivatives.
Setting this to zero tells us that $f$ is in the kernel of a strongly elliptic differential operator. In conjunction with the compactness of $M$, this implies that $f$ must be constant by the maximum principle.

\end{proof}
\end{lemma}

We will divide the proof of our main result Theorem \ref{thm-bc} into the following lemmas. In the first one we study the second order differential conditions in the characterisation \eqref{eq-char-harm-forms} of Bott-Chern and Aeppli harmonic forms.

\begin{lemma}\label{lemma-deldelbar}
Let $(M,J,\omega)$ be a compact almost K\"ahler manifold of real dimension $2n$. For any $k\in\N$ such that $2k\le n$, write any $(k,k)$-form $\psi$ as
\[
\psi=\sum_{m=0}^k \omega^{k-m}\wedge \alpha^{m,m},
\]
where every $(m,m)$-form $\alpha^{m,m}$ is primitive.
If $\psi$ satisfies both
\[
\del\delbar\psi=0\ \ \ \text{and}\ \ \  \del\delbar*\psi=0,
\]
then $\alpha^{0,0}\in\C$ is a complex constant. Moreover,
\[
\omega^{n-3}\wedge\del\delbar\alpha^{1,1}=0,
\]
i.e., $\del\delbar\alpha^{1,1}$ is primitive, and
\[
\omega^{n-4}\wedge\del\delbar\alpha^{2,2}=0.
\]
\end{lemma}
\begin{proof}
Fix $2k\le n$. Every form $\alpha^{m,m}$ is primitive, i.e.,
\begin{equation}\label{eq-primitive-alpha}
\omega^{n-2m+1}\wedge\alpha^{m,m}=0.
\end{equation}
By \eqref{*-primitive}, we have the formula
\[
*\psi=\sum_{m=0}^k(-1)^m\frac{(k-m)!}{(n-k-m)!}\omega^{n-k-m}\wedge \alpha^{m,m}.
\]
Now, assume that $\del\delbar\psi=\del\delbar*\psi=0$. Since $d\omega=0$, it follows
\begin{align}
\label{eq-del-delbar-psi}
&0=\del\delbar\psi=\sum_{m=0}^k\omega^{k-m}\wedge \del\delbar\alpha^{m,m},\\
&0=\del\delbar*\psi=\sum_{m=0}^k(-1)^m\frac{(k-m)!}{(n-k-m)!}\omega^{n-k-m}\wedge \del\delbar\alpha^{m,m}.\label{eq-del-delbar-star-psi}
\end{align}
We want to compare \eqref{eq-del-delbar-psi} and \eqref{eq-del-delbar-star-psi}. Note that
\[
k-m\le n-k-m \ \ \ \iff \ \ \ 2k\le n,
\]
therefore we can compute the wedge product between \eqref{eq-del-delbar-psi} and $\omega$ to the power $n-k-m-(k-m)=n-2k$ and obtain
\begin{equation}\label{eq-omega-del-delbar-psi}
0=\omega^{n-2k}\wedge\del\delbar\psi=\sum_{m=0}^k\omega^{n-k-m}\wedge \del\delbar\alpha^{m,m}.
\end{equation}

If we take the wedge product of $\omega^{k-1}$ with both equations \eqref{eq-del-delbar-star-psi} and \eqref{eq-omega-del-delbar-psi}, we find
\begin{align*}
0=\omega^{k-1}\wedge\del\delbar*\psi&=\sum_{m=0}^k(-1)^m\frac{(k-m)!}{(n-k-m)!}\omega^{n-m-1}\wedge \del\delbar\alpha^{m,m}\\
&=\frac{k!}{(n-k)!}\omega^{n-1}\wedge\del\delbar\alpha^{0,0}-\frac{(k-1)!}{(n-k-1)!}\omega^{n-2}\wedge\del\delbar\alpha^{1,1}
\end{align*}
by \eqref{eq-primitive-alpha}, since $n-2m+1\le n-m-1$ iff $m\ge 2$.
In the same way, 
\begin{align*}
0=\omega^{k-1}\wedge\omega^{n-2k}\wedge\del\delbar\psi&=\sum_{m=0}^k\omega^{n-m-1}\wedge \del\delbar\alpha^{m,m}\\
&=\omega^{n-1}\wedge\del\delbar\alpha^{0,0}+\omega^{n-2}\wedge\del\delbar\alpha^{1,1}
\end{align*}
by \eqref{eq-primitive-alpha}. Now, thanks to the last two equations, we easily deduce
\[
\omega^{n-1}\wedge\del\delbar\alpha^{0,0}=0
\]
and
\[
\omega^{n-2}\wedge\del\delbar\alpha^{1,1}=0.
\]
In particular, this yields that $\alpha^{0,0}\in \C$ is a complex constant by Lemma \ref{lemma elliptic}.

If we take the wedge product of $\omega^{k-2}$ with both equations \eqref{eq-del-delbar-star-psi} and \eqref{eq-omega-del-delbar-psi}, we find
\begin{align*}
0=\omega^{k-2}\wedge\del\delbar*\psi&=\sum_{m=1}^k(-1)^m\frac{(k-m)!}{(n-k-m)!}\omega^{n-m-2}\wedge \del\delbar\alpha^{m,m}\\
&=-\frac{(k-1)!}{(n-k-1)!}\omega^{n-3}\wedge\del\delbar\alpha^{1,1}+\frac{(k-2)!}{(n-k-2)!}\omega^{n-4}\wedge\del\delbar\alpha^{2,2}
\end{align*}
by \eqref{eq-primitive-alpha}.
In the same way, 
\begin{align*}
0=\omega^{k-2}\wedge\omega^{n-2k}\wedge\del\delbar\psi&=\sum_{m=1}^k\omega^{n-m-2}\wedge \del\delbar\alpha^{m,m}\\
&=\omega^{n-3}\wedge\del\delbar\alpha^{1,1}+\omega^{n-4}\wedge\del\delbar\alpha^{2,2}
\end{align*}
by \eqref{eq-primitive-alpha}.
Now, thanks to the last two equations, we easily deduce
\[
\omega^{n-3}\wedge\del\delbar\alpha^{1,1}=0
\]
and
\[
\omega^{n-4}\wedge\del\delbar\alpha^{2,2}=0,
\]
which is the claim.
\end{proof}

\begin{remark}
If we take the wedge product of $\omega^{k-l}$, for $3\le l \le k-1$, with both equations \eqref{eq-del-delbar-star-psi} and \eqref{eq-omega-del-delbar-psi}, we find similar sums, but this time we have three or more addends. This does not imply, in general, that every addend is equal to 0.
\end{remark}

In the next lemma we study the first order differential conditions in the characterisation \eqref{eq-char-harm-forms} of Bott-Chern harmonic forms.

\begin{lemma}\label{lemma-del}
Let $(M,J,\omega)$ be a compact almost K\"ahler manifold of real dimension $2n$. For any $k\in\N$ such that $2k\le n$, write any $(k,k)$-form $\psi$ as
\[
\psi=\sum_{m=0}^k \omega^{k-m}\wedge \alpha^{m,m},
\]
where every $(m,m)$-form $\alpha^{m,m}$ is primitive.
Assume that $\alpha^{0,0}\in\C$ is a complex constant. If $\del\psi=0$, then $\del\alpha^{1,1}$ is primitive, i.e., 
\[
\omega^{n-2}\wedge\del\alpha^{1,1}=0.
\]
If $\delbar\psi=0$, then $\delbar\alpha^{1,1}$ is primitive, i.e.,
\[
\omega^{n-2}\wedge\delbar\alpha^{1,1}=0.
\]
\end{lemma}
\begin{proof}
Fix $2k\le n$. 
Assume that $\del\psi=0$ and $\alpha^{0,0}\in\C$. Since $d\omega=0$, it follows
\begin{align}
&0=\del\psi=\sum_{m=1}^k\omega^{k-m}\wedge \del\alpha^{m,m}.\label{eq-del-psi}
\end{align}
If we take the wedge product of $\omega^{n-k-1}$ and \eqref{eq-del-psi}, we find
\begin{align*}
0=\omega^{n-k-1}\wedge\del\psi&=\sum_{m=1}^k\omega^{n-m-1}\wedge \del\alpha^{m,m}=\omega^{n-2}\wedge\del\alpha^{1,1},
\end{align*}
by \eqref{eq-primitive-alpha}.

Assume now that $\delbar\psi=0$ and $\alpha^{0,0}\in\C$. Since $d\omega=0$, it follows
\begin{align}
&0=\delbar\psi=\sum_{m=1}^k\omega^{k-m}\wedge \delbar\alpha^{m,m}.\label{eq-delbar-psi}
\end{align}
If we take the wedge product of $\omega^{n-k-1}$ and \eqref{eq-delbar-psi}, we find
\begin{align*}
0=\omega^{n-k-1}\wedge\delbar\psi&=\sum_{m=1}^k\omega^{n-m-1}\wedge \delbar\alpha^{m,m}=\omega^{n-2}\wedge\delbar\alpha^{1,1},
\end{align*}
by \eqref{eq-primitive-alpha}, and this ends the proof.
\end{proof}

\begin{remark}
If we take the wedge product of $\omega^{n-k-l}$, for $2\le l \le n-k-1$, with both equations \eqref{eq-del-psi} and \eqref{eq-delbar-psi}, we find similar sums, but this time we have two or more addends. This does not imply, in general, that every addend is equal to 0.
\end{remark}

Finally, in the following lemma we study the first order differential conditions in the characterisation \eqref{eq-char-harm-forms} of Aeppli harmonic forms.

\begin{lemma}\label{lemma-delstar}
Let $(M,J,\omega)$ be a compact almost K\"ahler manifold of real dimension $2n$. For any $k\in\N$ such that $2k\le n$, write any $(k,k)$-form $\psi$ as
\[
\psi=\sum_{m=0}^k \omega^{k-m}\wedge \alpha^{m,m},
\]
where every $(m,m)$-form $\alpha^{m,m}$ is primitive.
Assume that $\alpha^{0,0}\in\C$ is a complex constant. If $\del*\psi=0$, then $\del\alpha^{1,1}$ is primitive, i.e., 
\[
\omega^{n-2}\wedge\del\alpha^{1,1}=0.
\]
If $\delbar*\psi=0$, then $\delbar\alpha^{1,1}$ is primitive, i.e.,
\[
\omega^{n-2}\wedge\delbar\alpha^{1,1}=0.
\]
\end{lemma}
\begin{proof}
Fix $2k\le n$. 
By \eqref{*-primitive}, we have the formula
\[
*\psi=\sum_{m=0}^k(-1)^m\frac{(k-m)!}{(n-k-m)!}\omega^{n-k-m}\wedge \alpha^{m,m}.
\]
Assume that $\del*\psi=0$ and $\alpha^{0,0}\in\C$. Since $d\omega=0$, it follows
\begin{align}
&0=\del*\psi=\sum_{m=1}^k(-1)^m\frac{(k-m)!}{(n-k-m)!}\omega^{n-k-m}\wedge \del\alpha^{m,m}.\label{eq-del-star-psi}.
\end{align}
If we take the wedge product of $\omega^{k-1}$ and \eqref{eq-del-star-psi}, we find
\begin{align*}
0=\omega^{k-1}\wedge\del*\psi&=\sum_{m=1}^k(-1)^m\frac{(k-m)!}{(n-k-m)!}\omega^{n-m-1}\wedge \del\alpha^{m,m}\\
&=-\frac{(k-1)!}{(n-k-1)!}\omega^{n-2}\wedge\del\alpha^{1,1}.
\end{align*}
by \eqref{eq-primitive-alpha}, and this is equivalent to the first claim.

Now, assume that $\delbar*\psi=0$ and $\alpha^{0,0}\in\C$. Since $d\omega=0$, it follows
\begin{align}
&0=\delbar*\psi=\sum_{m=1}^k(-1)^m\frac{(k-m)!}{(n-k-m)!}\omega^{n-k-m}\wedge \delbar\alpha^{m,m}\label{eq-delbar-star-psi}.
\end{align}
If we take the wedge product of $\omega^{k-1}$ and \eqref{eq-delbar-star-psi}, we find
\begin{align*}
0=\omega^{k-1}\wedge\delbar*\psi&=\sum_{m=1}^k(-1)^m\frac{(k-m)!}{(n-k-m)!}\omega^{n-m-1}\wedge \delbar\alpha^{m,m}\\
&=-\frac{(k-1)!}{(n-k-1)!}\omega^{n-2}\wedge\delbar\alpha^{1,1},
\end{align*}
by \eqref{eq-primitive-alpha}, and this is equivalent to the second claim.
\end{proof}

We can now prove the following properties of Bott-Chern and Aeppli harmonic $(k,k)$-forms on a compact almost K\"ahler manifold.

\begin{theorem}\label{thm-bc}
Let $(M,J,\omega)$ be a compact almost K\"ahler manifold of real dimension $2n$. For any $k\in\N$, write any $(k,k)$-form $\psi$ as
\[
\psi=\sum_{m=0}^{\min(k,n-k)} \omega^{k-m}\wedge \alpha^{m,m},
\]
where every $(m,m)$-form $\alpha^{m,m}$ is primitive. If $\psi$ is Bott-Chern or Aeppli harmonic, then $\alpha^{0,0}\in\C$ is a complex constant, $\del\alpha^{1,1},\delbar\alpha^{1,1},\del\delbar\alpha^{1,1}$ are primitive, i.e.,
\[
\omega^{n-2}\wedge\del\alpha^{1,1}=\omega^{n-2}\wedge\delbar\alpha^{1,1}=\omega^{n-3}\wedge\del\delbar\alpha^{1,1}=0,
\]
 and
\[
\omega^{n-4}\wedge\del\delbar\alpha^{2,2}=0.
\]
\end{theorem}
\begin{proof}
Let us begin with the case $2k\le n$.
Note that if $\psi$ is Bott-Chern or Aeppli harmonic, then $\psi$ satisfies both
\[
\del\delbar\psi=0,\ \ \ \text{and}\ \ \ \del\delbar*\psi=0,
\]
and thus we can apply Lemma \ref{lemma-deldelbar}. Finally, if $\psi$ is Bott-Chern harmonic, we can apply Lemma \ref{lemma-del}. Conversely, if $\psi$ is Aeppli harmonic, we can apply Lemma \ref{lemma-delstar}. This concludes the proof of the case $2k\le n$.

Conversely, assume now that $2k\ge n$. By \eqref{*-primitive}, we have the formula
\[
*\psi=\sum_{m=0}^{n-k}(-1)^m\frac{(k-m)!}{(n-k-m)!}\omega^{n-k-m}\wedge \alpha^{m,m}.
\]
Set $l:=n-k$ and
\[
\beta^{m,m}:=(-1)^m\frac{(k-m)!}{(n-k-m)!}\alpha^{m,m}.
\]
We note that $2l\le n$ and the primitive decomposition of $*\psi\in A^{l,l}$ is
\[
*\psi=\sum_{m=0}^{l} \omega^{l-m}\wedge {\beta}^{m,m}
\]
By \eqref{bc-a-duality}, we know that $\psi$ is Bott-Chern harmonic iff $*\psi$ is Aeppli harmonic, and $\psi$ is Aeppli harmonic iff $*\psi$ is Bott-Chern harmonic. Therefore by the first part of the theorem applied to $*\psi$ we conclude that $\beta^{0,0}\in\C$, $\del\beta^{1,1},\delbar\beta^{1,1},\del\delbar\beta^{1,1}$ are primitive
 and
\[
\omega^{n-4}\wedge\del\delbar\beta^{2,2}=0.
\]
It is an observation that the same holds respectively for $\alpha^{0,0},\alpha^{1,1},\alpha^{2,2}$, ending the proof.
\end{proof}

Since the coefficient of $\omega^k$ is constant in the primitive decomposition of any Bott-Chern or Aeppli harmonic $(k,k)$-form, we can state the following characterisations of the spaces of Bott-Chern and Aeppli harmonic $(k,k)$-forms.

\begin{theorem}\label{cor bc decomp}
Let $(M,J,\omega)$ be a compact almost K\"ahler manifold of real dimension $2n$. For any $k\in \mathbb{N}$ we have
$$\mathcal{H}^{k,k}_{BC} = \mathbb{C}\, \omega^k \oplus \left( \mathcal{H}_{BC}^{k,k} \cap \ker L^{n-k} \right)$$
and
$$\mathcal{H}^{k,k}_{A} = \mathbb{C}\, \omega^k \oplus \left( \mathcal{H}_{A}^{k,k} \cap \ker L^{n-k} \right).$$
\end{theorem}
\begin{proof}
Let us consider the primitive decomposition of a Bott-Chern or Aeppli harmonic $(k,k)$-form $\psi$, i.e.,
\[
\psi=\sum_{m=0}^{\min(k,n-k)} \omega^{k-m}\wedge \alpha^{m,m}.
\]
Thanks to Theorem \ref{thm-bc}, we know that the coefficient of $\omega^k$, denoted by $\alpha^{0,0}$, is a complex constant. Since $\omega^{n-2m+1}\wedge\alpha^{m,m}=0$, therefore
\[
\omega^{n-k}\wedge\omega^{k-m}\wedge\alpha^{m,m}=\omega^{n-m}\wedge\alpha=0
\]
for any $m\ge1$. This proves the two inclusions $\subseteq$. The other two inclusions $\supseteq$ are trivial.
\end{proof}

\begin{remark}
When we consider the case of Theorem \ref{cor bc decomp} when $k = 1$, we recover the results of Tardini and the second author, \cite[Theorems 3.2 and 3.3]{piovani-tardini}. Namely, we have a complete decomposition into primitive harmonic forms
$$\mathcal{H}_{BC}^{1,1} = \mathbb{C}\, \omega \oplus \left(\mathcal{H}_{BC}^{1,1}\cap P^{1,1}\right), $$
$$\mathcal{H}_{A}^{1,1} = \mathbb{C}\, \omega \oplus \left(\mathcal{H}_{A}^{1,1}\cap P^{1,1}\right). $$
Corollary \ref{cor counterexample} will show that a complete decomposition into primitive harmonic forms does not hold in higher bidegrees $(k,k)$ with $k\ge 2$.
\end{remark}

The previous corollary can be further specialized in real dimension 8 for bidegree $(2,2)$.

\begin{corollary}\label{cor bc 8}
Let $(M,J,\omega)$ be a compact almost K\"ahler manifold of real dimension $8$. We have
\[
\H^{2,2}_{BC}=\C\,\omega^2\oplus\left\{\omega\wedge\alpha+\beta\,|\,\alpha\in P^{1,1},\,\beta\in P^{2,2},\,\omega\wedge\del\alpha+\del\beta=\delbar\alpha=\delbar\beta=0\right\},
\]
and
\[
\H^{2,2}_{A}=\C\,\omega^2\oplus\left\{\omega\wedge\alpha+\beta\,|\,\alpha\in P^{1,1},\,\beta\in P^{2,2},\,\omega\wedge\del\alpha-\del\beta=\delbar\alpha=\delbar\beta=0\right\}.
\]
\end{corollary}
\begin{proof}
Let us prove the Bott-Chern case. The Aeppli case is proved by a similar argument. By Theorem \ref{cor bc decomp}, we know
\[
\mathcal{H}^{2,2}_{BC} = \mathbb{C}\, \omega^2 \oplus \left( \mathcal{H}_{BC}^{2,2} \cap \ker L^{2} \right),
\]
therefore we want to prove that
\[
\mathcal{H}_{BC}^{2,2} \cap \ker L^{2}=\left\{\omega\wedge\alpha+\beta\,|\,\alpha\in P^{1,1},\,\beta\in P^{2,2},\,\omega\wedge\del\alpha+\del\beta=\omega\wedge\delbar\alpha=\delbar\beta=0\right\}.
\]
The inclusion $\supseteq$ is straightforward. Let us then take a form $\psi\in\mathcal{H}_{BC}^{2,2} \cap \ker L^{2}$. Its primitive decomposition is
\[
\psi=\omega\wedge\alpha+\beta,
\]
with $\alpha\in P^{1,1}$ and $\beta\in P^{2,2}$. Since $\psi$ is Bott-Chern harmonic, we know $\del\psi=\delbar\psi=0$, that is
\[
\omega\wedge\del\alpha+\del\beta=\omega\wedge\delbar\alpha+\delbar\beta=0.
\]
Moreover, by Theorem \ref{thm-bc}, we have
\[
\del\delbar\beta=0.
\]
Let us compute the pointwise inner product between $\omega\wedge\delbar\alpha$ and $\delbar\beta$. By the almost K\"ahler identities of \cite[Proposition 3.1]{cirici-wilson-2}, in particular $[\Lambda,\delbar]=-i\del^*$, we have
\begin{align*}
\langle L\delbar\alpha,\delbar\beta\rangle=\langle \delbar\alpha,\Lambda\delbar\beta\rangle=-i\langle \delbar\alpha,\del^*\beta\rangle=i\langle \delbar\alpha,*\delbar\beta\rangle.
\end{align*}
Now we integrate this pointwise inner product on the manifold to get the usual $L^2$ product between forms, obtaining
\begin{align*}
\la L\delbar\alpha,\delbar\beta\ra=i\la\delbar\alpha,*\delbar\beta\ra=i\la \alpha, \delbar^**\delbar\beta\ra=i\la\alpha,*\del\delbar\beta\ra=0.
\end{align*}
Since $L\delbar\alpha=-\delbar\beta$ and they are $L^2$ orthogonal, they must both be equal to zero. Now, by the Lefschetz isomorphism, $L\delbar\alpha=0$ if and only if $\delbar\alpha=0$. This ends the proof.
\end{proof}

\section{Primitive decomposition of Dolbeault Harmonic $(k,k)$-forms}

The next theorem yields similar conclusions for the Dolbeault case to the ones in the Bott-Chern and Aeppli case.

\begin{theorem}\label{thm dolbeault}
Let $(M,J,\omega)$ be a compact almost K\"ahler manifold of real dimension $2n$. Let $\psi$ denote a $(k,k)$-form, for some $k \in \mathbb{N}$. We can write
$$\psi = \sum_{m=0}^{\min(k,n-k)} \omega^{k-m}\wedge \alpha^{m,m},  $$
with $\alpha^{m,m} \in P^{m,m}$. If $\psi$ is Dolbeault harmonic then $\alpha^{0,0} \in \mathbb{C}$ is a complex constant and $\del\alpha^{1,1},\delbar\alpha^{1,1}$ are primitive.
\begin{proof}
We start by considering the case when $2k \leq n$, using a similar argument to the one used in \cite[Theorem 3.4]{cattaneo-tardini-tomassini}.

Note that $\psi$ is Dolbeault harmonic if and only it satisfies $\overline \partial \psi = 0$ and $ \partial * \psi = 0$. Since $\omega$ is almost K\"ahler, when we write these conditions out using the primitive decomposition of $\psi$ we get
\begin{align}
    0&=\delbar\psi=\sum_{m=0}^k \omega^{k-m}\wedge \overline \partial \alpha^{m,m},\label{eq Dolbeault del-bar}\\
0&=\del*\psi=\sum_{m=0}^k (-1)^m \frac{(k-m)!}{(n-k-m)!}\omega^{n-k-m}\wedge \partial \alpha^{m,m}.\label{eq Dolbeault del-star}
\end{align}

Then, by taking the wedge product of $\omega^{n-k-1}$ with equation \eqref{eq Dolbeault del-bar} we find that
\begin{align}
    \sum_{m=0}^k \omega^{n-m-1}\wedge \overline \partial \alpha^{m,m} =\omega^{n-1}\wedge \overline\partial\alpha^{0,0} + \omega^{n-2}\wedge \overline\partial\alpha^{1,1} = 0, \label{eq Dolbeault 1}
\end{align}
since $\omega^{n-2m+1} \wedge \alpha^{m,m}=0$ for all $m \in \mathbb{N}$. Similarly, by taking the wedge product of $\omega^{k-1}$ with equation \eqref{eq Dolbeault del-star} we find
\begin{align*}
    &\sum_{m=0}^k (-1)^m \frac{(k-m)!}{(n-k-m)!}\omega^{n-m-1}\wedge \partial \alpha^{m,m}=\\
    &=\frac{k!}{(n-k)!}\omega^{n-1}\wedge \partial\alpha^{0,0} - \frac{(k-1)!}{(n-k-1)!}\omega^{n-2}\wedge \partial\alpha^{1,1} = 0.
\end{align*}
Multiplying this by $\frac{(n-k-1)!}{(k-1)!}$, we have
\begin{equation}
    \frac{k}{n-k}\omega^{n-1}\wedge \partial\alpha^{0,0} - \omega^{n-2}\wedge \partial\alpha^{1,1}=0. \label{eq Dolbeault 2}
\end{equation}
We can then sum the equations \eqref{eq Dolbeault 1} and \eqref{eq Dolbeault 2}, to get
\begin{equation}
    \omega^{n-1}\wedge\left(\overline\partial + \frac{k}{n-k}\partial\right) \alpha^{0,0} + \omega^{n-2}\wedge\left(\overline \partial - \partial\right) \alpha^{1,1} = 0.\label{eq Dolbeault combined 1}
\end{equation}

Now, by making use of the operator
$$d^c = i(\mu - \partial + \overline \partial - \overline\mu) $$
and the fact that $\mu = \overline \mu = 0$ when acting on an $(n-1,n-1)$-form, we can rewrite equation \eqref{eq Dolbeault combined 1} as 
\begin{equation*}
      \omega^{n-1} \wedge\left(\overline\partial + \frac{k}{n-k}\partial\right) \alpha^{0,0} = id^c \left(\omega^{n-2}\wedge \alpha^{1,1}\right).
\end{equation*}
Applying $-id^c$ the right hand side vanishes and we are left with
\begin{align*}
    \omega^{n-1}\wedge(\mu - \partial + \overline \partial - \overline\mu)\left(\overline\partial + \frac{k}{n-k}\partial\right) \alpha^{0,0} = \left(\frac{k}{n-k}+1\right)\omega^{n-1}\wedge \partial \overline \partial \alpha^{0,0} = 0.
\end{align*}
In particular, we have $\omega^{n-1} \wedge \partial \overline\partial \alpha^{0,0}=0$, which implies that $\alpha^{0,0}$ is constant by Lemma \ref{lemma elliptic}. Now, looking at \eqref{eq Dolbeault combined 1}, we deduce that $\del\alpha^{1,1}$ and $\delbar\alpha^{1,1}$ are primitive.

The result in the case when $2k \geq n$ follows simply from the first case by Serre duality. Namely, we have $\mathcal{H}_{\overline \partial}^{k,k} = *\overline{\mathcal{H}_{\overline \partial}^{n-k,n-k}}$. Therefore, any element of $\mathcal{H}_{\overline \partial}^{k,k}$ with $2k \geq n$ can be written as
$$*\overline{\psi} = \sum_{m=0}^{n-k}(-1)^m\frac{(n-k-m)!}{(k-m)!}\omega^{k-m}\wedge \overline{\alpha^{m,m}}  $$
for some $\psi \in \mathcal{H}_{\overline\partial}^{n-k,n-k}$, and since $2(n-k) \leq n$ we conclude that $\alpha^{0,0}$ is constant and $\del\alpha^{1,1}$ and $\delbar\alpha^{1,1}$ are primitive.
\end{proof}
\end{theorem}
\begin{remark}
Note that the same statement of Theorem \ref{thm dolbeault} also holds for $\Delta_\del$-harmonic $(k,k)$-forms. The proof of this is equivalent to the one above up to conjugation.
\end{remark}

The above result allows the primitive decomposition of $(k,k)$-forms to descend partially to Dolbeault harmonic $(k,k)$-forms in the following way, using the same proof as Theorem \ref{cor bc decomp}.
\begin{theorem}\label{cor Dolbeault partial decomp}
Let $(M,J,\omega)$ be a compact almost K\"ahler manifold of real dimension $2n$. For any $k\in \mathbb{N}$ we have
$$\mathcal{H}^{k,k}_{\overline \partial} = \mathbb{C}\, \omega^k \oplus \left( \mathcal{H}_{\overline\partial}^{k,k} \cap \ker L^{n-k} \right)$$
and
$$\mathcal{H}^{k,k}_{\partial} = \mathbb{C}\, \omega^k \oplus \left( \mathcal{H}_{\partial}^{k,k} \cap \ker L^{n-k} \right).$$
\end{theorem}
\begin{remark}
When we consider the case of Theorem \ref{cor Dolbeault partial decomp} with $k = 1$, we recover the results of Cattaneo, Tardini and Tomassini, \cite[Theorem 3.4 and Corollary 3.5]{cattaneo-tardini-tomassini}. Namely, we have a complete decomposition into primitive harmonic forms
$$\mathcal{H}_{\overline\partial}^{1,1} = \mathbb{C}\, \omega \oplus \left(\mathcal{H}_{\overline\partial}^{1,1}\cap P^{1,1}\right), $$
$$\mathcal{H}_{\partial}^{1,1} = \mathbb{C}\, \omega \oplus \left(\mathcal{H}_{\partial}^{1,1}\cap P^{1,1}\right). $$
Corollary \ref{cor counterexample} will show that a complete decomposition into primitive harmonic forms does not hold in higher bidegrees $(k,k)$ with $k\ge 2$.
\end{remark}

If we now restrict to real dimension 8, we obtain the following.
\begin{corollary}\label{cor dol 8}
Let $(M,J,\omega)$ be a compact almost K\"ahler manifold of real dimension $8$. We have
\[
\H^{2,2}_{\delbar}=\C\,\omega^2\oplus\left\{\omega\wedge\alpha+\beta\,|\,\alpha\in P^{1,1},\,\beta\in P^{2,2},\,\omega\wedge\delbar\alpha+\delbar\beta=\omega\wedge\del\alpha-\del\beta=0\right\},
\]
and
\[
\H^{2,2}_{\del}=\C\,\omega^2\oplus\left\{\omega\wedge\alpha+\beta\,|\,\alpha\in P^{1,1},\,\beta\in P^{2,2},\,\omega\wedge\del\alpha+\del\beta=\omega\wedge\delbar\alpha-\delbar\beta=0\right\}.
\]
\end{corollary}
\begin{proof}
The result follows immediately from Theorem \ref{cor Dolbeault partial decomp}, from the characterisation \eqref{eq-char-harm-forms} of Dolbeault and $\del$-harmonic forms and from formula \eqref{*-primitive}.
\end{proof}

From Corollaries \ref{cor bc 8} and \ref{cor dol 8}, we deduce the following inclusions of the spaces of harmonic forms in dimension 8.
\begin{corollary}\label{cor incl}
Let $(M,J,\omega)$ be a compact almost K\"ahler manifold of real dimension $8$. We have
\[
\H^{2,2}_{BC}\subseteq\H^{2,2}_{\del},\ \ \ \ \ \ \ \H^{2,2}_{A}\subseteq\H^{2,2}_{\delbar}.
\]
\end{corollary}

\section{Examples}

In this section we present two 8-dimensional examples of nilmanifolds and study their harmonic $(2,2)$-forms.

\begin{example}\label{ex torus}
We consider a similar construction to the one in \cite[Example 4.9]{cattaneo-tardini-tomassini}. Let $\mathbb{T}^8=\Z^8\backslash \R^8$ be the $8$-dimensional torus with real coordinates $(x^1,y^1,x^2,y^2,x^3,y^3,x^4,y^4)$ on $\R^8$. Let $g=g(x^4,y^4)$ be a non constant function on $\mathbb{T}^8$. We define an almost complex structure $J$ by setting 
\begin{align*}
&\phi^1=e^gdx^1+ie^{-g}dy^1,\\
&\phi^2=dx^2+idy^2,\\
&\phi^3=dx^3+idy^3,\\
&\phi^4=dx^4+idy^4.
\end{align*}
to be a global coframe of $(1,0)$-forms.
Denote by $V_1,V_2,V_3,V_4$ the global frame of vector fields dual to $\phi^1,\phi^2,\phi^3,\phi^4$. Then, the structure equations are
\begin{align*}
&d\phi^1=V_4(g)\phi^{4\c1}-\c{V_4}(g)\phi^{\c1\c4},\\
&d\phi^2=0,\\
&d\phi^3=0,\\
&d\phi^4=0.
\end{align*}
We endow $(\mathbb{T}^8,J)$ with the almost K\"ahler metric given by the fundamental form
\[
\omega=i(\phi^{1\c1}+\phi^{2\c2}+\phi^{3\c3}+\phi^{4\c4}).
\]
We consider the volume form 
\[
\frac{\omega^4}{4!}=\phi^{1\c12\c23\c34\c4}=\phi^{1234\c1\c2\c3\c4}.
\]

We want to show that the inclusion $\H^{2,2}_{BC}\subseteq\H^{2,2}_{\del}$ of Corollary \ref{cor incl} is strict. Let us consider the form $\phi^{12\c2\c3}$. We compute
\begin{align*}
&\del\phi^{12\c2\c3}=0,\\
&\delbar*\phi^{12\c2\c3}=\delbar\phi^{14\c3\c4}=0,\\
&\delbar\phi^{12\c2\c3}=V_4(g)\phi^{4\c12\c2\c3}\ne0.
\end{align*}
Therefore, $\phi^{12\c2\c3}\in\H^{2,2}_{\del}\setminus\H^{2,2}_{BC}$, proving our claim. We note that the same holds for the form $\phi^{13\c2\c3}$. By duality (see \eqref{eq-char-harm-forms}), also note that $*\psi\in\H^{2,2}_{\delbar}\setminus\H^{2,2}_{A}$.

We also want to show that in general the primitive decomposition of $(2,2)$-forms does not descend to the spaces of Bott-Chern, Aeppli, Dolbeault and $\del$-harmonic forms. Namely, we want to find a $(2,2)$-form
\[
\psi=\omega\wedge\alpha+\beta,
\]
where $\alpha\in P^{1,1}$ and $\beta\in P^{2,2}$, such that $\alpha$ and $\beta$ are not Bott-Chern and $\del$-harmonic. Considering $\c\psi$, the same can then be shown for the cases of Aeppli and Dolbeault. Let us consider the form $\psi=2\phi^{2\c14\c4}$. Its primitive decomposition is
\[
2\phi^{2\c14\c4}=(\phi^{2\c14\c4}+\phi^{2\c13\c3})+(\phi^{2\c14\c4}-\phi^{2\c13\c3}),
\]
where
\begin{align*}
&\phi^{2\c14\c4}+\phi^{2\c13\c3}=-i\omega\wedge\phi^{2\c1}\in L(P^{1,1}),\\
&\phi^{2\c14\c4}-\phi^{2\c13\c3}\in P^{2,2}.
\end{align*}
Set $\beta=\phi^{2\c14\c4}-\phi^{2\c13\c3}$ and $\alpha=-i\phi^{2\c1}$. Then we have
\begin{align*}
&\delbar\alpha=0,\\
&\delbar\beta=0,\\
&\omega\wedge\del\alpha+\del\beta=0,\\
&\del\alpha=-i\c{V_4}(g)\phi^{21\c4}\ne0,\\
&\del\beta=-\c{V_4}(g)\phi^{21\c43\c3}\ne0,
\end{align*}
therefore $\psi\in\H^{2,2}_{BC}\cap\H^{2,2}_{\del}$, while $\alpha\notin\H^{1,1}_{BC}\cup\H^{1,1}_{\del}$ and $\beta\notin\H^{2,2}_{BC}\cup\H^{2,2}_{\del}$. Note that this also shows that the two results of Corollary \ref{cor bc 8} cannot be strengthened by asking that $\omega\wedge\del\alpha=\del\beta=0$, instead of $\omega\wedge\del\alpha+ \del\beta=0$ or $\omega\wedge\del\alpha- \del\beta=0$.
\end{example}

Summing up the results from the above example, we state the following corollary.
\begin{corollary}\label{cor counterexample}
There exists a compact almost K\"ahler manifold $(M,J,\omega)$ of real dimension $8$ such that
\[
\H^{2,2}_{BC}\not\supseteq\H^{2,2}_{\del},\ \ \ \ \ \ \ \H^{2,2}_{A}\not\supseteq\H^{2,2}_{\delbar}
\]
and
\[
\H^{2,2}_{D}\not\subseteq \C\,\omega^2\oplus L\left(P^{1,1}\cap\H^{1,1}_{D}\right)\oplus \left(P^{2,2}\cap \H^{2,2}_{D}\right),
\]
where $D\in\{BC,A,\delbar,\del\}$.
\end{corollary}

We remark that the almost K\"ahler structure of Example \ref{ex torus} is not left invariant with respect to the usual Lie group structure of the torus. In fact, we do not have any example of an $8$-dimensional manifold with a left invariant almost K\"ahler structure which satisfies the conditions of Corollary \ref{cor counterexample}. Below we present one such example of a manifold (described in \cite{cattaneo-tardini-tomassini}, Example 4.3), with a left invariant almost K\"ahler structure. We show that in this example all left invariant harmonic forms $\psi \in \mathcal{H}^{2,2}_{D}$ have a primitive decomposition such that each component is also contained in $\mathcal{H}^{2,2}_{D}$, where $D \in \{BC, A, \partial, \overline \partial\}$.   

\begin{example}
We start by defining 
$$\mathbb{H}(1,2) := \left\{\begin{pmatrix}
1 & 0 & x_1 & z_1 \\
0 & 1 & x_2 & z_2 \\
0 & 0 & 1   & y   \\
0 & 0 & 0   & 0
\end{pmatrix}  \, \middle| \, x_1, x_2, y, z_1, z_2 \in \mathbb{R}\right\}. $$
Then, if we let $\Gamma \subset \mathbb{H}(2,1)$ be the subgroup of elements with integer valued entries, we can define the compact 8-manifold $X := \Gamma \backslash \mathbb{H}(2,1) \times \mathbb{T}^3$. A left invariant coframe on $X$ can be given by
$$e^1 = dx_2, \quad \quad e^2 = dx_1, \quad \quad e^3 = dy, \quad \quad e^4 = du, $$
$$e^5 = dz_1 - x_1 dy, \quad \quad e^6 = dz_2 - x_2 dy, \quad \quad e^7 = dv, \quad \quad e^8 = dw, $$
where $u, v, w$ parametrise $\mathbb{T}^3$. 
From this coframe, we derive the structure equations
$$de^1 = de^2 = de^3= de^4 = de^7 = de^8 = 0 $$
$$de^5 = -e^{23}\quad \quad \quad \quad  de^6 = - e^{12}. $$
An almost Hermitian structure is then defined so that
$$\phi^1 = e^1 + i e^5 \quad \quad \quad \quad \phi^2 = e^2 + i e^6$$
$$\phi^3 = e^3 + i e^7 \quad \quad \quad \quad \phi^4 = e^4 + i e^8$$
are orthonormal $(1,0)$-forms, with structure equations given by
$$d\phi^{1} = -\frac i4 \left(\phi^{23}+ \phi^{2\overline 3} - \phi^{3\overline 2} + \phi^{\overline 2 \overline 3}\right), $$
$$d\phi^2 = -\frac i4 \left( \phi^{13} + \phi^{1 \overline 3} - \phi^{3\overline 1} + \phi^{\overline 1 \overline 3}\right), $$
$$d\phi^3 = d\phi^4 = 0. $$
The fundamental form corresponding to this almost Hermitian structure is
$$\omega = i\left(\phi^{1\overline 1} + \phi^{2\overline 2} + \phi^{3\overline 3} + \phi^{4\overline 4}\right). $$
This fundamental form is $d$-closed and so the structure we have defined is almost K\"ahler.

It is then a trivial (although computationally tedious) task to compute the space of left invariant forms contained in  $\mathcal{H}_{\overline\partial}^{2,2}$. This space is spanned by the following sixteen $(2,2)$-forms
\begin{align*}
    &\phi^{1 2 \overline 1 \overline 2}, &&\phi^{1 2 \overline 1 \overline 3}+ \phi^{1 3 \overline 1 \overline 2}, &&&\phi^{1 2 \overline 3 \overline 4}+ \phi^{1 3 \overline 2 \overline 4}, &&&&\phi^{1 2 \overline 2 \overline 3}+ \phi^{2 3 \overline 1 \overline 2}, \\ &\phi^{1 3 \overline 2 \overline 3}+ \phi^{2 3 \overline 1 \overline 3}, && \phi^{1 3 \overline 2 \overline 4}+\phi^{2 3 \overline 1 \overline 4}, &&&\phi^{1 3 \overline 1 \overline 3}+\phi^{2 3 \overline 2 \overline 3}, &&&&\phi^{1 3 \overline 1 \overline 4}+\phi^{2 3 \overline 2 \overline 4},\\ &\phi^{1 4 \overline 2 \overline 3}+\phi^{2 4 \overline 1 \overline 3}, &&\phi^{1 4 \overline 2 \overline 4}+\phi^{2 4 \overline 1 \overline 4}, &&&\phi^{1 4 \overline 1 \overline 3}+ \phi^{2 4 \overline 2 \overline 3}, &&&&\phi^{1 4 \overline 1 \overline 4}+ \phi^{2 4 \overline 2 \overline 4},\\ &\phi^{2 4 \overline 1 \overline 3} +\phi^{3 4 \overline 1 \overline 2}, &&\phi^{1 4 \overline 3 \overline 4}+ \phi^{3 4 \overline 1 \overline 4}, &&&\phi^{2 4 \overline 3 \overline 4}+ \phi^{3 4 \overline 2 \overline 4}, &&&&\phi^{3 4 \overline 3 \overline 4}.   
\end{align*}
Furthermore, when we compute the spaces of left invariant forms contained in $\mathcal{H}^{2,2}_{\partial}, \mathcal{H}^{2,2}_{BC}$ and $\mathcal{H}^{2,2}_{A}$, we find that these spaces are all equal to the space of left invariant forms in $\mathcal{H}^{2,2}_{\overline \partial}$. We claim that this implies that the primitive decomposition descends to harmonic $(2,2)$-forms in all of these spaces, i.e., if $\psi \in \mathcal{H}^{2,2}_{D}$ is left invariant then
$$\psi \in  \C\,\omega^2\oplus L\left(P^{1,1}\cap\H^{1,1}_{D}\right)\oplus \left(P^{2,2}\cap \H^{2,2}_{D}\right),  $$
for any $D \in \{\partial, \overline\partial, BC, A\}$.

To see why this is the case, consider the $(2,2)$-form $\psi \in A^{2,2}(X)$ with primitive decomposition
$\psi = c\omega^2 + \alpha \wedge \omega + \beta,$
where $c \in \mathbb{C}, \alpha \in P^{1,1}, \beta \in P^{2,2}$. Let $\psi$ be contained in both $\mathcal{H}^{2,2}_{\partial}$ and $\mathcal{H}^{2,2}_{\overline\partial}$. From $\psi \in \mathcal{H}^{2,2}_{\overline \partial}$ we obtain
\begin{align*}
&\overline\partial \alpha \wedge \omega + \overline\partial \beta = 0,\\
&\partial \alpha \wedge \omega -\partial \beta = 0,
\end{align*}
and from $\psi \in \mathcal{H}^{2,2}_{\partial}$ we obtain
\begin{align*}
&\overline\partial \alpha \wedge \omega - \overline\partial \beta = 0,\\
&\partial \alpha \wedge \omega +\partial \beta = 0.
\end{align*}
Combining these results, we see that 
$$\partial \alpha \wedge \omega = \overline \partial \alpha \wedge \omega = \partial \beta = \overline \partial \beta = 0.   $$
This is sufficient to imply that $\omega\wedge \alpha, \beta \in \mathcal{H}^{2,2}_{D}$ or in other words
$$\psi \in  \C\,\omega^2\oplus L\left(P^{1,1}\cap\H^{1,1}_{D}\right)\oplus \left(P^{2,2}\cap \H^{2,2}_{D}\right) $$
for all $D \in \{\partial, \overline \partial, BC, A\}$.  

\end{example}

\end{document}